\documentclass[a4paper,12pt,reqno]{amsart}

\usepackage{amsmath,amsfonts,amsthm,amssymb, latexsym,euscript,dsfont,graphicx,units,color}
\usepackage[latin1]{inputenc}
\usepackage[font=footnotesize]{subfig}
\def\ie{{\em i.e.},~}

\newcommand{\hp}{\hat p}

\newcommand{\ld}{L^2}

\newcommand{\1}{{\bf 1}}

\newcommand{\E}{\mathbb E}
\newcommand{\G}{\mathbb G}

\newcommand{\R}{\mathbb R}
\newcommand{\N}{\mathbb N}

\newcommand{\cf}{\mathcal F}
\newcommand{\cg}{\mathcal G}

\newcommand{\cn}{\mathcal N}

\newcommand{\al}{\alpha}
\newcommand{\ep}{\varepsilon}

\newcommand{\ga}{\gamma}

\newcommand{\ka}{\kappa}
\newcommand{\la}{\lambda}

\newcommand{\om}{\omega}
\newcommand{\oom}{\Omega}

\newcommand{\si}{\sigma}

\newcommand{\lp}{\left(}
\newcommand{\rp}{\right)}
\newcommand{\lc}{\left[}
\newcommand{\rc}{\right]}
\newcommand{\lcl}{\left\{}
\newcommand{\rcl}{\right\}}
\newcommand{\lln}{\left|}
\newcommand{\rrn}{\right|}
\newcommand{\lla}{\left\langle}
\newcommand{\rra}{\right\rangle}

\newcommand{\bean}{\begin{eqnarray*}}
\newcommand{\eean}{\end{eqnarray*}}
\newcommand{\ben}{\begin{enumerate}}
\newcommand{\een}{\end{enumerate}}
\newcommand{\beq}{\begin{equation}}
\newcommand{\eeq}{\end{equation}}

\newtheorem{theorem}{Theorem}[section]

\newtheorem{hypothesis}[theorem]{Hypothesis}
\newtheorem{lemma}[theorem]{Lemma}

\newtheorem{proposition}[theorem]{Proposition}

\theoremstyle{remark}
\newtheorem{remark}[theorem]{Remark}

\begin{document}
\title[peeling wavelet denoising algorithm]{Convergence and performances of the peeling wavelet denoising algorithm}
\author[C. Lacaux\and A. Muller \and  R. Ranta \and S. Tindel]{C\'eline Lacaux \and Aur\'elie Muller \and Radu Ranta \and Samy Tindel}
\date{\today}

\address{
{\it Céline Lacaux, Aurélie Muller, Samy Tindel:} {\rm Institut {\'E}lie Cartan Nancy, B.P. 239,
54506 Vand{\oe}uvre-l{\`e}s-Nancy, France}. {\it Email:} {\tt
[lacaux, muller, tindel]@iecn.u-nancy.fr}
\newline
$\mbox{ }$\hspace{0.1cm}
{\it Radu Ranta:} {\rm Centre de Recherche en Automatique de Nancy, 2 Avenue de la Forêt de Haye, 54500 Vand{\oe}uvre-l{\`e}s-Nancy, France}. {\it Email:} {\tt
Radu.Ranta@ensem.inpl-nancy.fr}
 }



\date{\today}

\begin{abstract}
This note is devoted to an analysis of the so-called peeling algorithm in wavelet denoising.
Assuming that the wavelet coefficients of the signal can be modeled by generalized Gaussian random
variables, we compute a critical thresholding constant for the algorithm, which depends on the
shape parameter of the generalized Gaussian distribution. We also quantify the optimal number of
steps which have to be performed, and analyze the convergence of the algorithm. Several versions
of the obtained algorithm were implemented and tested against classical wavelet denoising
procedures on benchmark and simulated biological signals.
\end{abstract}

\maketitle

\section{Introduction}\label{Intro}

Among the wide range of applications of wavelet theory which have emerged during the last 20
years, the processing of noisy signals is certainly one of the most important one. Especially
attractive to the community has been the \emph{thresholding} algorithm, and the great amount of
efforts in this direction is well represented by the enthusiastic discussion in \cite{DJKP}, by
the application oriented presentation \cite{ABS} or by the sharp uniform central limit theorems in \cite{GN}. This fundamental algorithm can be summarized in
the following way: recall that the wavelet decomposition of a function $z\in L^2(\R)$ is usually
written as: \beq\label{eq:dcp-g-wave} z(t)=\sum_{k=0}^{2^{j_0}-1}\al_{j_0 k}\phi_{j_0 k}(t) +
\sum_{j=j_0}^{\infty}\sum_{k=0}^{2^{j}-1}\beta_{j k}\psi_{j k}(t), \eeq where the coefficients
$\al,\beta$ are obtained by projection in $\ld(\R)$:
$$
\al_{j_0 k}=\lla \phi_{j_0 k}, z \rra_{\ld(\R)},
\quad\mbox{and}\quad
\beta_{j k}=\lla \psi_{j k}, z \rra_{\ld(\R)}.
$$
The functions $\psi$ and $\phi$ are respectively called mother and father wavelets, and enjoy some
suitable scaling and algebraic properties (see e.g. \cite{Da,Mal} for a complete account on
wavelet decompositions). In this context, the thresholding algorithm assumes that, if $z$ can be
decomposed into $z=x+n$, where $x$ is the useful signal and $n$ its noisy part, then the wavelet
coefficients corresponding to $n$ will typically be very small. A reasonable estimation for the
signal $x$ is thus:
$$
\hat x (t)=\sum_{k=0}^{2^{j_0}-1}\al_{j_0 k}\, \1_{\{|\al_{j_0 k}|\ge \tau\}}\, \phi_{j_0 k}(t)
+ \sum_{j=j_0}^{J}\sum_{k=0}^{2^{j}-1}
\beta_{j k}\, \1_{\{|\beta_{j k}|\ge \tau\}}\, \psi_{j k}(t),
$$
where $\tau$ is a suitable threshold (which may also depend on the resolution $j$ and, in
practice, is often null for the coefficients $\al_{j_0}$ of the father wavelet / scale function)
and where $J$ corresponds to the maximal resolution one is allowed to consider. It is then proved
in the aforementioned references \cite{ABS,DJKP} that this kind of estimator satisfies some nice
properties concerning the asymptotic behavior of the approximation error, in terms of the total
number of wavelet coefficients (which is denoted by $N$ in the sequel).

\smallskip

One of the drawbacks of the thresholding algorithm is that it may also spoil the original signal
$x$. The critical issue is the value of the threshold(s) $\tau$: too low it is inefficient, too
high it distorts the information from $x$. In order to improve the performances of wavelet-based
denoising algorithms by adapting them to the processed signals, the following iterative method,
called \emph{peeling} algorithm, has been introduced and shown to be particularly useful for
biomedical applications in \cite{CW,HP}.
%One of the drawbacks of the thresholding algorithm,  which has been discovered on an experimental
%basis, is that it may also spoil the original signal $x$. This is especially true when the wavelet
%decomposition of $x$ is sparse, namely when only a small number of wavelet coefficients
%$\al,\beta$ have a high intensity.
%In order to improve the performances of wavelet-based denoising
%algorithms in this kind of situation (and in particular for biomedical applications), the
%following iterative method, called \emph{peeling} algorithm, has been introduced in \cite{CW,HP}.
It still relies on an a priori decomposition of the observed signal $z$ into $z=x+n$, where $x$ is
the signal itself, and $n$  is a noise. The algorithm intends then to separate $x$ from $n$
iteratively, and the $k\textsuperscript{th}$ step of the procedure produces an estimated signal
$x_k$, and a noise $n_k$, initialized for $k=0$ as $n_0=z$. These functions will always be
assimilated with the vector of their wavelet coefficients. Then the  $(k+1)\textsuperscript{th}$
step is as follows:
\begin{enumerate}
\item Compute $\si_k^2=\frac{\| n_k\|^2}{N}$, where we recall that $N$ denotes the total
    number of wavelet coefficients involved in the analysis.
\item Set a thresholding level $T_{k+1}$ as $T_{k+1}=h(\si_k)$, where $h$ is usually linear,
    which means that $T_{k+1}= F \, \si_k$ for a certain coefficient $F$.
\item Compute $\Delta x_{k+1}$ as:
    $$
    \Delta x_{k+1}(q)=n_{k}(q) \, \1_{\{|n_{k}(q)|\ge T_{k+1}\}},
    $$
    for all the coefficients $q$ of the wavelet decomposition. The vectors $x_{k+1}, n_{k+1}$
    are then defined as $x_{k+1}=x_{k}+\Delta x_{k+1}$, and $n_{k+1}=n_{k}-\Delta x_{k+1}$.
\item Loop this procedure until a stop criterion of the form $\| n_k\|^2-\| n_{k+1}\|^2 \leq
    \ep$ is reached, for a certain positive constant $\ep$. Notice that one can choose
    $\ep=0$.
\end{enumerate}
This iterative procedure tends to retrieve a higher quantity of (approximate) signal $x$ from the
noisy input $z$, correcting some of the failures of the original thresholding algorithm in some
special situations.

\smallskip

On the basis of these promising experimental results, the peeling algorithm has been further
investigated in \cite{RHLW,RHLW2}, and it has been first observed in those references that  the
peeling problem could be handled through a fixed point algorithm. This possibility stems basically
from the fact that the sequence $\{T_k;\, k\ge 0 \}$ is decreasing (as $\| n_k\|^2\geq\|
n_{k+1}\|^2$), which means that the previous algorithm can be reduced to the following:
\begin{enumerate}
\item Set $T_0=+\infty$ and $T_{k+1}=f_N(T_k)$, where $f_N$ is of the form: \beq\label{eq:2a}
    f_N(x)=F \ \lc \frac{\sum_{q\le N}z^2(q)\, \1_{\{|z(q)|< x\}}}{N}  \rc^{1/2}. \eeq For a
    suitable constant $F$, this defines a converging decreasing sequence $(T_k)$, such that
    $\lim_{k\to\infty}T_k=T_f>0$.
\item
Stop the loop when $T_{k+1}=T_f$, and then set
\beq\label{eq:2}
\hat x(q)=z(q) \, \1_{\{|z(q)|\ge T_{f}\}}.
\eeq
\end{enumerate}
It is shown in \cite{RHLW} that this algorithm is almost surely convergent, and a further analysis
of the coefficient $F$ is performed in \cite{RHLW2}.

\smallskip

However, in spite of the efforts made in the aforementioned references \cite{RHLW,RHLW2}, a
probabilistic analysis of the algorithm is still missing. The current article proposes to make a
step in this direction, and we proceed now to describe the results we have obtained. First of all,
let us say a few words about the model we have chosen for our signal $z$. This signal is of course
characterized by the family of its wavelets coefficients, which will be denoted from now on by
$\{z(q);\, q\le N\}$, and it is usual in signal processing to model these coefficients by
independent generalized Gaussian variables (see e.g. \cite{PZP}), all defined on a common complete
probability space $(\oom,\cf,P)$. To be more specific, we will assume the following:
\begin{hypothesis}\label{hyp:1.1}
The  wavelet coefficients $\{z(q);\, q\le N\}$ of our signal $z$ form an i.i.d family of
generalized Gaussian variables, whose common density $(p_{\si,u}(x))_{x\in\R}$ is given by
\beq\label{eq:gg-density}
 p_{\si,u}(x)=\alpha e^{-|\beta x|^u}, \quad\mbox{with}\quad
 \beta=\frac{1}{\si} \lp\frac{\Gamma(3/u)}{\Gamma(1/u)}\rp^{1/2}, \quad \alpha =\frac{\beta u}{2\Gamma(1/u)},
 \eeq
 where $\Gamma$ stands for the usual Gamma function $\Gamma(\xi)=\int_0^\infty e^{-x}x^{\xi-1}dx$. Notice that the coefficient $\si>0$ above is the standard deviation of each random variable $z(q)$, and that $u>0$ represents the shape parameter of the probability law ($u=2$ for the Gaussian, $u=1$ for the Laplace pdf).
\end{hypothesis}
It should be stressed at this point that this model does not take into account the possible
decomposition of $z$ into a signal plus a noise, since we model directly the wavelet coefficients
of $z$. It is however suitable for the main example we have in mind, namely a situation where the
family $\{z(q);\, q\le N\}$ is sparse. Indeed, when $u<2$ in expression~ (\ref{eq:gg-density}),
the distribution of the $z(q)$'s becomes heavy tailed, which means that one expects a few large
coefficients and many small ones. This is the situation we are mostly interested in, but our
analysis below is valid for any coefficient $u>0$ once our basic model is assumed to be realistic.
It should also be stressed that in the end, our algorithm is also of thresholding type, as may be
seen from equation (\ref{eq:2}). This means in particular that it is certainly suitable to
retrieve signal from a noisy input, on the same basis as the original thresholding algorithm.

\smallskip

With these preliminary considerations in mind, here are the two main results which will be
presented in this paper:

\smallskip

\noindent \textbf{(1)} We have seen that the sequence of thresholds $\{T_k;\, k\ge 0\}$ involved
in the peeling algorithm converges almost surely. However, it is easily checked that it can
converge either to a strictly positive quantity $T_f$, either to 0. This latter limit is not
suitable for our purposes, since it means that no noise will be extracted from our signal. One of
the main questions raised by the peeling algorithm is thus to find an appropriate constant $F$ in
(\ref{eq:2a}) such that \textit{(i)} The algorithm yields a convergence to a non trivial threshold
$T_f>0$. \textit{(ii)} $F$ is small enough, so that a sufficient part of the original signal is
retrieved.

\smallskip

The previous attempts in this direction were simply (see \cite{HP}) to take $F=3\si$ with
experimental arguments; after the analysis performed in \cite{RHLW2}, this quantity was reduced to
$F=F_m$, a quantity which is defined by
\beq \label{eq:Fm}%
F_m=\sqrt{\frac{3\Gamma(1/u)}{u}(ue)^{1/u}}. \eeq However, the latter bound has been obtained
thanks to some rough estimates, and we have thus decided here to go one step further into this
direction. Indeed, our first task will be to determine precisely, and on a mathematical ground, a
constant $F_c=F_c(u,\si)$ such that: if $F>F_c$, the algorithm yields a convergence, with high
probability, to a strictly positive constant $T_f=T_f(\omega)$, whose fluctuations around a typical non-random value $x^*$ will be determined. In particular, we will see that our constant $F_c$ is always
lower than $F_m$. Whenever $F<F_c$, we also show that $T_k$ converges to 0 with high probability
(see Proposition~\ref{Asy:F<Fc}).

\smallskip

\noindent \textbf{(2)} In the regime $F>F_c$, we determine that the optimal number of steps for
the peeling algorithm is of order $\log(N)$, where we recall that $N$ is the total number of
wavelet coefficients involved in the analysis. After this optimal number of steps, Theorem
\ref{thm:cvgce-noisy-dyn} quantifies also sharply the oscillations of $T_f$ with respect to its
theoretical value $m_f$.

\smallskip

It is important to show that our theoretical results can really be applied to real data. We have
thus decided first to compare the performances of our algorithm with other wavelet denoising
procedures, on some classical benchmark signals proposed in \cite{DJ94}. It will be seen that our
algorithm performs well with respect to other methods, independently of the value of the shape
parameter in (\ref{eq:gg-density}) and of the form of the benchmark signal. Interestingly enough,
this assertion is true even if Hypothesis \ref{hyp:1.1} is not always satisfied by the benchmark
signals under consideration.

\smallskip

A second step in our practical part of the study is the following: since the peeling algorithm has
been introduced first in a medical context, we give an illustration of its performances on ECG
type signals. More specifically, we shall consider a simulated ECG signal, and observe the
denoising effect of our algorithm on a perturbed version of those electrocardiograms. It will be
observed again that the algorithm under analysis is a good compromise between denoising and
preservation of the original signal.

\smallskip

Let us mention some open problems that have been left for a subsequent publication: first, let us
recall that the so-called block thresholding has improved the behavior of the original
thresholding algorithm in a certain number of situations (see e.g. \cite{Ch} for a nice overview).
It would be interesting to analyze the effect of this procedure in our peeling context. In
relation to this problem, one should also care about some reasonable dependence structure among
wavelet coefficients, beyond the independent case treated in this article. Finally, we have
assumed in this paper that the parameters of the distribution $p_{\si,u}$ were known, which is
typically not true in real world applications. One should thus be able to quantify the effect of
parameter estimation on the whole denoising process.

\smallskip

Here is how our article is structured: we show how to compute optimal constants for the peeling
algorithm at Section \ref{sec:critic-constants}. Then the probabilistic analysis of the algorithm
is leaded at Section \ref{sec:proba-analysis}. Finally, some numerical experiment on simulated and
pseudo-real data are performed at Section 4.

\section{Critical constants for the peeling algorithm}\label{sec:critic-constants}
This section is devoted to the computation of an optimal constant $F$ in equation (\ref{eq:2a}),
ensuring a convergence of the threshold $T_k$ to a non trivial $T_f$, and still allowing to
retrieve a maximal amount of approximate signal from our noisy input $z$.

\smallskip

Let us start this procedure by changing slightly the setting of the peeling algorithm. Indeed, it
will be essential for our convergence theorems at Section \ref{sec:proba-analysis}, to be able to
express the fixed point algorithm in terms of  empirical processes. To this purpose, we resort to
a simple change of variables by setting:
$$
U_k=T_k^2, \quad\mbox{and}\quad Y(q)=z(q)^2.
$$
Note that $U_0=+\infty$. It is then readily checked that the fixed point algorithm of
Section~\ref{Intro} is equivalent to the following:
\begin{enumerate}
\item
$U_{k+1}=g_N(U_k)$, where $g_N$ is of the form:
$$
g_N(x)=\frac{F^2}{N}\sum_{q\le N}Y(q)\, \1_{\{Y(q)< x\}} .
$$
For a suitable constant $F$, this defines a converging decreasing sequence $(U_k)$, such
that $\lim_{k\to\infty}U_k=U_f$.
\item
Stop the loop when $U_{k+1}=U_f$, and then set
$$
\hat x(q)=z(q) \, \1_{\{|z(q)|\ge \sqrt{U_{f}}\}}.
$$
\end{enumerate}
The fixed point $U_f=U_f(\omega)$ is then solution of the equation $g_N(x)=x$. We wish to find the
critical (minimal) $F$ which ensures $U_f$ to be strictly positive.

\smallskip

A preliminary step towards this aim is to consider a natural deterministic problem related to the
equation $g_N(x)=x$. Indeed, for a fixed value of $x\in\R_+$, the law of large numbers asserts
that the random variable $g_N(x)$ converges almost surely to the quantity
$$
g_{\si,u}(x)= F^2\int_0^x w \hat{p}_{\si,u}(w) \, dw,
$$
where $\hp_{\si,u}$ is the common density of the random variables $Y(q)=z(q)^2$, given explicitly
under Hypothesis \ref{hyp:1.1} by \beq\label{eq:def-hap-p}
\hp_{\si,u}(w)=\frac{p_{\si,u}(\sqrt{w})}{\sqrt{w}}\1_{\{w>0\}} =\alpha \frac{1}{\sqrt
w}e^{-(\beta \sqrt{w})^u}\1_{\{w>0\}}. \eeq This result is only a simple convergence result, and not
an almost sure uniform convergence of $g_N$ towards $g_{\si,u}$. However, as will be shown in the
Section \ref{sec:proba-analysis}, the fixed point $U_f$ is close in some sense to a fixed point of
$g_{\si,u}$. Therefore, our study of the fixed points of $g_N$ can be reduced to the study of
$g_{\si,u}$. Our aim is now to give some sharp conditions on the coefficient $F$ ensuring that the
equation $g_{\si,u}(x)=x$ has at least one solution $x>0$.

\smallskip

According to  (\ref{eq:def-hap-p}) we obtain, for $x>0$:
$$g_{\si,u}(x)= F^2\alpha \int_0^x\sqrt{w}e^{-(\beta \sqrt w)^u}dw,$$
with  $\beta=\frac{1}{\si} \lp\frac{\Gamma(3/u)}{\Gamma(1/u)}\rp^{1/2}$ and $\alpha =\frac{\beta u}{2\Gamma(1/u)}.$
%where we note $g_{\si,u}$ instead of $g_\si$ to enhance the shape parameter $u$.
Furthermore, a simple change of variables argument yields:%
\beq\label{eq:g-sigma-u}
g_{\si,u}(x)=F^2\sigma^2\Gamma_{{\rm inc}}((\beta\sqrt x)^u,3/u)%
\eeq
where $\Gamma_{{\rm inc}}(x,a)=\frac{1}{\Gamma(a)}\int_{0}^x e^{-t}t^{a-1}dt$ is the incomplete
Gamma function.

Observe that the trivial change of variables $w=\si^2 y$ in the integral above yields the
expression: \beq\label{eq:relation-f-sigma-b} g_{\si,u}(x)= \,\si^2g_{1,u}(x/\si^2) \eeq Hence,
solving $g_{\si,u}(x)=x$ is equivalent to solve $g_{1,u}(v)=v$, for $v=x/\si^2$. We shall consider
our equation in this reduced form, since $\sigma$ has only a scale role in the fixed point problem
of $g_{\si,u}$ and can be omitted in the study. In the sequel, we thus solve the problem in its
reduced form: $g_{1,u}(x)=x$. Furthermore, for notational sake, we simply denote $g_{1,u}$ by $g$.

\smallskip

The resolution of the equation $g(x)=x$ boils down to the joint study of $g$ and of a function $d$
defined by $d(x)=g(x)-x$. These studies are a matter of elementary considerations, and it is
easily deduced that $g$ has the following form, as a function from $\R_+$ to $\R_+$:
\ben %
\item
$g$ is increasing and $\lim_{x\to\infty}g(x)=F^2 \triangleq g_\infty$.
%\item
%$g$ has exactly two fixed points apart from 0, called $\ell_1$ and $x^*$, with $\ell_1<x^*$.
\item $g$ is convex on $[0,\beta^{-2}u^{-2/u}]$ and concave on $[\beta^{-2}u^{-2/u},+\infty)$,
    where  $\beta=\left(\frac{\Gamma(3/u)}{\Gamma(1/u)}\right)^{1/2}$.
\een %
More precisely, it is easy to prove the existence of a critical value $F_c$ such that:
\begin{enumerate}
\item If $F<F_c$, the only fixed point of $g$ is 0.
\item If $F=F_c$, $g$ has exactly two fixed points (0 and $x^*_c>\beta^{-2}u^{-2/u}$).
\item If $F>F_c$, $g$ has exactly three fixed points (0, $l_1$, and $x^*$), such that
    $0<l_1<x^*$ and $\beta^{-2}u^{-2/u}<x^*$.
\end{enumerate}
These facts are well illustrated by Figure 1 (for $\si=1$).
\begin{figure}[htbp]
\begin{center}
	\includegraphics[scale=0.6]{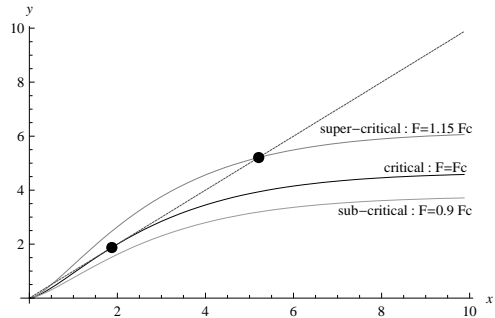}
\caption{Curves corresponding to $g$, in the critical case ($F=F_c$), and in supercritical and subcritical cases ($F=1.15 F_c$ and $F=0.9 F_c$), for $\si=1$ and $u=2$.}
\end{center}
\label{pointfixe}
\end{figure}

\smallskip

Let us turn now to the computation of the critical coefficient $F_c$ and the critical fixed point
$x^*_c$. In fact, once the study of our function $d$ is performed, it is also easy to show that
$F\equiv F_c$ and $r\equiv x^*_c$ are solutions of the system:
$$
g'(r)=1, \quad\mbox{and}\quad g(r)=r,
$$
where we recall that the coefficient $F$ enters into the definition of $g$. This system is
equivalent, in the generalized Gaussian case, to:
$$
\begin{cases}
F^2\alpha\sqrt r e^{-(\beta\sqrt r)^u}-1&=0\\
F^2\Gamma_{\rm inc}((\beta\sqrt r)^u,3/u) -r&=0,
\end{cases}
$$
where it should be reminded that $\Gamma$ and $\Gamma_{\rm inc}$ designate respectively Gamma and
incomplete Gamma functions. The latter system can be solved with the Mathematica software, and the
solutions for different $u$ are illustrated in Figure \ref{fig2}.
\begin{figure}[htbp]
\begin{center}
		\includegraphics[scale=0.6]{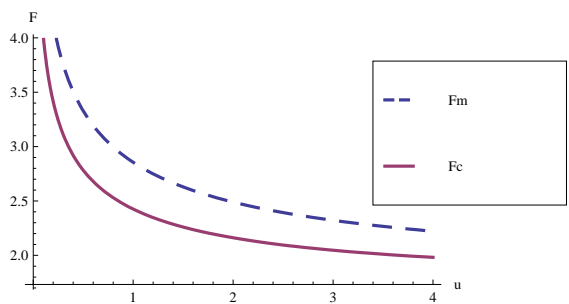}
\caption{The values of critical $F_c$ and $F_m$ for $\si=1$ and $u\in[0,4]$.}
\label{fig2}
\end{center}
\end{figure}
Some typical values of $F_c$ in terms of $u$ are also given in Table \ref{tab:u-Fc}.
\begin{table}
\begin{tabular}{l | c c c c c c}
$u$&0.1&0.5&1&2&3&4 \\ \hline
$F_c$& 4.0215& 2.7830&2.42537&2.16169&2.0472&1.98181
\end{tabular}\caption{Critical constant $F_c$ for different shapes $u$.\label{tab:u-Fc}}
\end{table}
In particular, it can be observed that $F_c$ is smaller than the bound $F_m$ proposed by
\cite{RHLW2}, which has been recalled at equation (\ref{eq:Fm}).

\section{Probabilistic analysis of the algorithm}\label{sec:proba-analysis}

\subsection{Comparison Noisy dynamics/ Deterministic dynamics}
The exact dynamics governing the
sequence $\{ U_n;\, n\ge 0 \}$ is of the form $U_{n+1}=g_N(U_n)$. In order to compare
this with the deterministic dynamics, let us recast this relation into:
$$
U_{n+1}=g(U_n) +\ep_{n,N},
\quad\mbox{where}\quad
\ep_{n,N}=g_N(U_n)-g(U_n).
$$
Notice that the errors $\ep_{n,N}$ are far from being independent, which means that
the relation above does not define a Markov chain. However, a fairly simple expression
is available for~$U_n$:
\begin{proposition}\label{prop:expression-noisy-dynamics}
For $n\ge 0$, set $g^{\circ n}$ for the $n$\textsuperscript{th} iteration of $g$. Then, for $n\ge 0$,
we have:
$$
U_{n}= g^{\circ n}(U_0)+ R_{n},
\quad\mbox{with}\quad
R_{n}=\sum_{p=0}^{n-1}\ep_{p,N} \prod_{q=2}^{n-p} g'(C_{p+q}),
$$
where the random variable $C_{j}$ ($j\ge 2$) is a certain real number within the interval
$[g^{\circ (j-1)}(U_0);$  $U_{j-1}]$. In the definition of $R_{n}$, we have also
used the conventions $\prod_{q=2}^{1}a_q=1$ and $R_0=0$.
\end{proposition}

\begin{proof}
It is easily seen inductively that $R_0=0$, $R_{1}=\ep_{0,N}$ and for $n\ge 1$
$$
R_{n+1}=g'(C_{n+1})R_{n}+\ep_{n,N}.
$$
Hence, by a backward induction, we obtain:
$$
R_{n}=\sum_{j=1}^{n}\ep_{n-j,N} \prod_{l=0}^{j-2} g'(C_{n-l})
=\sum_{p=0}^{n-1}\ep_{p,N} \prod_{q=2}^{n-p} g'(C_{p+q}),
$$
which ends the proof.
\end{proof}

A useful property of the errors $\ep_{p,N}$ is that they concentrate exponentially fast (in terms
of $N$) around 0. This can be quantified in the following:
\begin{lemma}\label{Samy}
Assume that the wavelets coefficients are distributed according to a generalized Gaussian random variable with parameter $u>0$, whose density is given by (\ref{eq:gg-density}), and recall that $F$ is defined by equation (\ref{eq:2a}).  Then for every  $0<\gamma<(\beta/F)^u$, there exists  a  finite positive constant $K>0$ such that for all $N\ge 1$ and  for all $\la\in [0,\ga N^{u/4}]$, %there exists $K=K_\la>0$ such that
\beq\label{eq:exp-bnd-epsilon}
\mathbb{E}\left[ \textup{e}^{\la  | \ep_{p,N}|^{u/2} }\right] \le K.
\eeq
Moreover, for all $N\ge 1$, for all $p\ge 0$ and $l>0$,
\beq\label{eq:exp-bnd-proba-epsilon}
\mathbb{P}\lp \lln \ep_{p,N}\rrn \ge l\rp \le K^{}\textup{e}^{-\ga l^{u/2} N^{u/4}}.
\eeq
%Then there exits $\gamma>0$ such that for all $p\ge 0$ and $l>0$,
%\beq\label{eq:exp-bnd-proba-epsilon}
%\mathbb{P}\lp \lln \ep_{p,N}\rrn \ge l\rp \le \textup{e}^{-\ga l N^{u/4}}.
%\eeq
%Moreover, for all $\la\in [0,\ga N^{u/4})$, there exists $K=K_\la>0$ such that
%\beq\label{eq:exp-bnd-epsilon}
%\mathbb{E}\left[ \textup{e}^{\la  | \ep_{p,N}| }\right] \le K.
%\eeq
\end{lemma}

\begin{proof}
Recall that $\ep_{p,N}$ is defined by:
$$
\ep_{p,N}=g_N(U_p)-g(U_p)=\frac{F^2}{N}
\lp\sum_{j=1}^{N}Y(j)\, \1_{\{Y(j)< U_p\}}-g(U_p)\rp,
$$
for a collection $\{Y(i);\, i\le N\}$ of i.i.d random variables, where $Y(i)$ can be written as
$Y(i)=z(i)^2$ and $z(i)$ is a generalized Gaussian random variable with parameter $u>0$, whose
density is given by (\ref{eq:gg-density}). For a fixed positive $x$, the fluctuations
$g_N(x)-g(x)$ are easily controlled thanks to the classical central limit theorem or large
deviations principle. The difficulty in our case arises from the fact that $U_p$ is itself a
random variable, which rules out the possibility of applying those classical results. However,
uniform central limit theorems and deviation inequalities have been thoroughly studied, and our
result will be obtained by translating our problem in terms of empirical processes like in
\cite{VW}.

\smallskip

In order to express $\ep_{p,N}$ in terms of empirical processes, consider $x\in [0,\infty]$ and define
$h_x:\R_+\to\R_+$ by $h_x(u)=F^2u\, \1_{\{ u < x\}}$. Next, for $f:\R_+\to\R$, set
$$
\G_N f=\frac{1}{N^{1/2}} \sum_{i=1}^{N} \lc f(Y(i)) -  \E[f(Y(i))]\rc,
$$
and with these notations in mind, notice that
$$
\G_N h_x=\frac{1}{N^{1/2}} \sum_{i=1}^{N} \lc h_x(Y(i)) -  g(x)\rc.
$$
It is now easily seen that
$$
\ep_{p,N}=N^{-1/2}\G_N h_{U_p},
$$ and the key to our result will be to get good control on $\G_N h_x$ in terms of $N$, uniformly in $x\in [0,\infty]$.

\smallskip
Let us consider the class of functions $\cg=\{h_x;\, x\in[0,+\infty]\}$. According to the terminology of
\cite{VW}, the uniform central limit theorems are obtained when $\cg$ is a {\emph{Donsker}} class
of functions. A typical example of Donsker setting is provided by some
 %the
VC classes (see \cite[Section 2.6.2]{VW}).
 The VC classes can be briefly described as sets of functions whose subgraphs can only shatter a
finite collection of points, with a certain maximal cardinality $M$, in $\R^2$. For instance, the
collections of indicators
$$
\cf=\lcl  \1_{[0,x)}; \, x\in [0,+\infty]\rcl.
$$
is  a VC class.
Thanks to \cite[Lemma 2.6.18]{VW}, our class $\cg$ is also of VC type, since it can be written as
$$\cg=\cf\cdot h=\left\{fh; f\in \cf\right\},
$$
where  $h:\R_+\to\R_+$ is defined by $h(u)=h_{\infty}(u)=F^2u$.% and the
%Let us consider then a certain class of function $\cf$, the main application we have in mind being the class $\cg=\{g_x;\, x\in\R_+\}$. According to the terminology of \cite{VW}, the uniform central limit theorems are obtained when $\cf$ is a {\emph Donsker} class of functions. A typical example of Donsker setting is provided by the VC classes (see \cite[Section 2.6.2]{VW}), which can be briefly described as sets of functions whose subgraphs can only shatter a finite collection of points, with a certain maximal cardinality $M$, in $\R^2$. For instance, if $\cf$ is a collection of indicator functions, then $\cf$ is VC. Thanks to \cite[Lemma 2.6.18]{VW}, our class $\cg$ is also of VC type, since it can be written as $\cf\cdot h$, where $h$ is the identity function on $\R_+$ and
%$$
%\cf=\lcl  \1_{[0,x]}; \, x\in \R_+\rcl.
%$$

\smallskip

In order to state our concentration result, we still need to introduce the envelope $
\overline{\cg}$ of $\cg$, which is a function $\overline{\cg}:\R_+\to\R$ defined as
$$
\overline{\cg}(u)=\sup\{f(u);\, f\in\cg \},\ u\in\mathbb{R}_+.
$$
Note that in our particular example of application, we simply have $
\overline{\cg}=h$. %, where $h$ is the identity function on $\mathbb{R}_+$.
%which is defined as $\overline{\cg}(u)=\sup\{f(u);\, f\in\cg\}=h\PAR{u}$. In our standing application, the envelope of $\cg$ is simply the function $h=\1_{\R_+}$ introduced above.
Let us also introduce the following notation:
$$
\cn[\G_N;\cg,\la,m]\equiv \E^*\lc e^{\la\sup_{f\in \cg}|\G_N f|^m} \rc,
\quad\mbox{and}\quad
\cn[h;\la,m]\equiv \E\lc e^{\la |h(Y)|^m} \rc,
$$
where $\E^*$ is the  outer expectation (defined in \cite{VW} for measurability issues), $Y$ is the
square of a generalized Gaussian random variable with parameter $u>0$, $\la>0$ and $m\ge 0$.

\smallskip

Then, since $\cg$ is a VC class with measurable envelope, $\cg$ is a Donsker class and \cite[Theorem 2.14.5 p. 244]{VW} leads to:%a simple translation of \cite[Theorem 2.14.5 p. 244]{VW} yields:
$$
\cn[\G_N;\cg,\la,m] \le c \, \cn[h; \la,m],
$$
with $c$ a finite positive constant which does not depend on $N,\la$ and $\cg$. %, and a certain $\ga<\hat \ga$.
Furthermore, since $Y$ is the square of a generalized Gaussian random variable with parameter $u$,
it is readily checked that
$$
\cn[h; \la,m]<\infty
$$ for
$\la$ small enough (namely $\la < (\beta/F)^u$) and $m=u/2$. Recalling now that $\ep_{p,N}=N^{-1/2}\G_N h_{U_p}$, we have obtained:
$$
\mathbb{E}\left[ \textup{e}^{\la  | N^{1/2}\ep_{p,N}|^{u/2} }\right] \le \cn[\G_N;\cg,\la,u/2]\le c\cn[h; \ga,u/2]=K<\infty
$$
for $\lambda\le \ga<(\beta/F)^u$, which easily implies our claim (\ref{eq:exp-bnd-epsilon}). \\%The proof of (\ref{eq:exp-bnd-proba-epsilon}) can be done along the same lines, and is omitted here for sake of conciseness.

Let $l>0$. Then,
$$
\mathbb{P}\left(|\ep_{p,N}|\ge l\right)=\mathbb{P}\left(\textup{e}^{\ga N^{u/4} |\ep_{p,N}|^{u/2}}\ge \textup{e}^{\ga l^{u/2}N^{u/4}}\right).
$$
The concentration property (\ref{eq:exp-bnd-proba-epsilon}) is thus an easy consequence of
(\ref{eq:exp-bnd-epsilon}) Markov's inequality.

\end{proof}

\subsection{Supercritical case: $\mathbf{F>F_c}$}
%\subsection{Some notations about the function $g_{1,u}\equiv g$}\label{sec:notations-g}
%In light of the results shown at Section \ref{sec:fixed-point-gaussian},
In this section, we assume %from now
that $F>F_c$. Then it is easily deduced from the variations of $d$ given above that $g_{1,u}\equiv
g$ has the following form, as a function from $\R_+$ to $\R_+$ (see Figure 1): \ben
\item
$g$ is increasing and $\lim_{x\to\infty}g(x)=F^2 \triangleq g_\infty$.
\item
$g$ has exactly two fixed points apart from 0, called $\ell_1$ and $x^*$, with $\ell_1<x^*$.
\item
There exists $\ell_2\in (\ell_1,x^*)$ such that $g'(l_2)\le 1$ and $g$ is concave on $[\ell_2,\infty)$.
\item Let $\delta>0$ such that $l_1<l_2+\delta<x^*.$ Then, $g(l_2+\delta)>l_2+\delta$.
\een

With these properties
%hypothesis
in mind, we can study the convergence of the deterministic sequence $\{x_n;\, n\ge 0\}$ defined
recursively by $x_0=\infty$ and $x_{n+1}=g(x_n)$. Indeed, it is easily checked that $x_n$ is
decreasing to $x^*$ as $n\to\infty$. Furthermore, let $M_{x^*}=\sup\{g'(x);\, x\ge x^*\}=g'(x^*)$,
and recall that $M_{x^*}<1$. Then
$$
|x_{n+1}-x^*|=x_{n+1}-x^*=g(x_n)-g(x^*)\le M_{x^*} \lp x_{n}-x^* \rp,
$$
which means that a geometric convergence occurs: inductively, it is readily checked that,
for $n\ge 1$:
\beq\label{eq:geom-cvgce-xn}
|x_n-x^*|\le M_{x^*}^{n-1} \lp  g_\infty-x^*\rp,
\eeq
where we recall that $g_\infty=\lim_{x\to\infty}g(x)$.

\smallskip

We are now ready to prove the convergence result for the peeling algorithm, in terms of a
concentration result for the noisy dynamics around the deterministic one:
\begin{theorem}\label{thm:cvgce-noisy-dyn}
Assume $F>F_c$ (where these quantities are defined at Section \ref{sec:critic-constants}) and that
the wavelets coefficients are distributed according to a generalized Gaussian random variable with
parameter $u>0$, whose density is given by (\ref{eq:gg-density}). Let $\al<1/2$,
$C=\frac{-1}{\ln\lp M_{x^*}\rp}+\eta$ with $\eta>0$. For any $N\in \N^*$, let $n=n(N)=\left[C\al
\ln N\right]+1$. Then,  there exist  $A,\widetilde{\ga}$ two positive finite constants such that
for all $N\in\N^*$, and any $F$ lying in an arbitrary compact interval $[0,F_0]$, we have
$$
\mathbb{P}\lp \lln U_n-x^*\rrn\ge N^{-\al} \rp \le A\textup{e}^{-\widetilde{\ga} N^{(1/2-\al)u/2}}.
$$
\end{theorem}

\begin{remark}\label{rmk:cvgce-noisy-dyn}
This theorem induces three kind of information about the convergence of our algorithm: {\it (i)}
For a fixed number of wavelet coefficients $N$, the optimal number of iterations $n$ for the
peeling algorithm is of order $\ln(N)$. {\it (ii)} Once $n$ is fixed in this optimal way, $U_n$ is
close to the fixed point $x^*$ of $g$, the magnitude of $|U_n-x^*|$ being of order $N^{-(1/2-\ep)}$
for any $\ep>0$.  {\it (iii)} The deviations of $U_n$ from $x^*$ are controlled exponentially in
probability.
\end{remark}

\begin{proof}[Proof of Theorem \ref{thm:cvgce-noisy-dyn}]
Observe first that, owing to Proposition \ref{prop:expression-noisy-dynamics} and inequality
(\ref{eq:geom-cvgce-xn}), we have
$$
\lln U_n-x^*\rrn =\lln g^n\lp U_0\rp-x^*-R_n\rrn\le M_{x^*}^{n-1}(g_\infty-x^*)+\lln R_n \rrn,
$$
for any $n\ge 1$. Let then $\hat{\delta}>0$ and let us fix $n\ge 1$ such that
\begin{equation}
\label{controle-n}
M_{x^*}^{n-1}(g_\infty-x^*)\le \frac{\hat{\delta}}{2},
\end{equation}
i.e. $n\ge 1+ \ln(\hat{\delta}/(2g_\infty-2x^*))/\ln( M_{x^*})$. Then it is readily checked that:
\beq %
\label{eq:bnd1-Un-x} \mathbb{P}\lp \lln U_n-x^*\rrn\ge \hat{\delta} \rp \le \mathbb{P}\lp
\lln R_n\rrn\ge \frac{\hat{\delta}}{2}  \rp,
\eeq %
and we will now bound the probability in the right hand side of this inequality. To this purpose,
let us introduce a little more notation: for $n,k\ge 1$, let $\Omega_k$ be the set defined by
$$
\Omega_k=\lcl \om\in\Omega; \, \inf\lcl j\ge 0\,/\, U_j\le \ell_2+\delta\rcl=k\rcl,
$$
%where we recall that $\ell_2$ has been defined at Section \ref{sec:notations-g},
and  set also
$
\widetilde{\Omega}_n=\bigcup_{k=1}^n \Omega_k.
$
Then we can decompose (\ref{eq:bnd1-Un-x}) into:
\begin{equation}
\label{majE}
\mathbb{P}\lp \lln U_n-x^*\rrn\ge \hat{\delta} \rp \le \mathbb{P}\lp \widetilde{\Omega}_n \rp +\mathbb{P}\lp  \widetilde{\Omega}_n^c\cap\lcl \lln R_n\rrn \ge \frac{\hat{\delta}}{2}\rcl  \rp.
\end{equation}
We will now treat these two terms separately:

\smallskip

\noindent
{\it Step 1: Upper bound for} $\mathbb{P}( \widetilde{\Omega}_n )$.
Let us fix $k\ge 1$ and first study $\mathbb{P}\lp \Omega_k \rp$. To this purpose, observe first that
$$
\Omega_k\subset \lcl U_{k}\le l_2+\delta<U_{k-1} \rcl.
$$
Recall %from Section \ref{sec:notations-g}
that $\ell_2+\delta$ satisfies $g(\ell_2+\delta)>\ell_2+\delta$. %It also satisfies $g'(\ell_2+\delta)<1$.
Hence, since $U_{k}=g_N(U_{k-1})$ and invoking the fact that $g$ is an increasing function, the
following relation holds true on $\Omega_k$:
$$
g_N(U_{k-1})\le l_2+\delta
\quad\mbox{and}\quad
g(l_2+\delta )<g(U_{k-1}).
$$
We have thus proved that
$$
\Omega_k\subset \lcl g_N(U_{k-1})-g(U_{k-1})\le l_2+\delta-g(l_2+\delta)\rcl,
$$
where $l_2+\delta-g(l_2+\delta)\equiv -L<0.$ Since $g_N(U_{k-1})-g(U_{k-1})=\ep_{k-1,N}$ by
definition, we end up with:
$$
\mathbb{P}(\Omega_k)\le \mathbb{P}\lp \lln \ep_{k-1,N}\rrn \ge L\rp.
$$
A direct application of Lemma \ref{Samy} yields now the existence of $\ga,K\in (0,\infty)$ such
that for all $k\ge 1$ and all $N\ge 1$
$$
\mathbb{P}(\Omega_k)\le K \textrm{e}^{-\ga L^{u/2} N^{u/4}}.
$$
Hence
\begin{equation}
\label{Borne1}
\mathbb{P}(\widetilde{\Omega}_n)\le \sum_{k=1}^n
\mathbb{P}(\Omega_k)\le K n\textrm{e}^{-\ga L^{u/2} N^{u/4}}.
\end{equation}

\smallskip

\noindent {\it Step 2: Upper bound for} $\mathbb{P}(  \widetilde{\Omega}_n^c\cap\{ \lln R_n \rrn
\ge \frac{\hat{\delta}}{2}\}  ).$ We have constructed the set $\widetilde{\Omega}_n$ so that, for
all $2\le k\le n+1$, the random variables $C_k$ introduced at Proposition
\ref{prop:expression-noisy-dynamics} satisfy $0\le g'\lp C_k\rp \le \rho <1$ on
$\widetilde{\Omega}_n^c$. Thus
\begin{equation}
\label{cont-B2}
\mathbb{P}\lp  \widetilde{\Omega}_n^c\cap\lcl \lln R_n \rrn \ge \frac{\hat{\delta}}{2}\rcl  \rp\le
\mathbb{P}\lp \sum_{p=0}^{n-1} \lln \ep_{p,N}\rrn \rho^{n-1-p}\ge  \frac{\hat{\delta}}{2}\rp
\le \mathbb{P}\lp \sum_{p=0}^{n-1} \lln \ep_{p,N}\rrn \nu_p\ge  M_{n,\hat{\delta}}\rp,
\end{equation}
where we have set
$$
\nu_p=\frac{\rho^{n-1-p}(1-\rho)}{1-\rho^n},
\quad\mbox{and}\quad
M_{n,\hat{\delta}} =\frac{\hat{\delta}(1-\rho)}{2(1-\rho^n)},
$$
so that $\{\nu_p;\, 0\le p\le n-1\}$ is a probability measure on $\{0,\ldots,n-1\}$.

\smallskip

We introduce now a convex non-decreasing function $a_u$ which only depends on the shape parameter
$u$, and which behaves like $\exp(x^{u/2})$ at infinity. Specifically, if $u\ge 2$, we simply
define $a_u$ on $\mathbb{R}_+$ by
$$
a_u(x)=\textup{e}^{x^{u/2}}.
$$
When $u<2$, setting $s_u=\left(2/u-1\right)^{2/u}$, then $x\mapsto\exp(x^{u/2})$ is concave on
$[0,s_u]$ and convex on  $[s_u,+\infty)$. Then, we modify a little the definition of $a_u$ in
order to obtain a convex function: we set
\begin{equation}\label{eq:def-a-u}
a_u(x)=\textup{e}^{x^{u/2}} \1_{[s_u,\infty)} + \textup{e}^{s_u^{u/2}} \1_{[0,s_u)}
\end{equation}
where $s_u=\left(2/u-1\right)^{2/u}$.

\smallskip

Since $a_u$ is a non-decreasing  function, for all $\la > 0$, relation \eqref{cont-B2} implies
that:
\begin{eqnarray*}
\mathbb{P}\lp  \widetilde{\Omega}_n^c\cap\lcl \lln R_n \rrn \ge \frac{\hat{\delta}}{2}\rcl  \rp
&\le&
 \mathbb{P}\lp a_u \lp \la\sum_{p=0}^{n-1} \lln \ep_{p,N}\rrn \nu_p \rp
\ge a_u \lp \la  M_{n,\hat{\delta}}\rp\
\rp  \\
&\le& \frac{1}{a_u\lp\la M_{n,\hat{\delta}}\rp}\mathbb{E}\lc a_u\lp \la  \sum_{p=0}^{n-1}| \ep_{p,N}| \nu_p \rp\rc,
\end{eqnarray*}
where we have invoked Markov's inequality for the second step. Hence, applying Jensen's inequality, for all $\la>0$,
 % Lemma~\ref{Samy},
 we obtain:
$$
\displaystyle
\mathbb{P}\lp  \widetilde{\Omega}_n^c\cap\lcl \lln R_n \rrn\ge \frac{\hat{\delta}}{2}\rcl  \rp\le \frac{1}{a_u\left(\la M_{n,\hat{\delta}}\right)}\sum_{p=0}^{n-1}\nu_p\mathbb{E}\lp a_u\lp\la  | \ep_{p,N}|\rp \rp.%\le K\textup{e}^{-\la M_{n,\hat{\delta}}},
$$
Furthermore, owing to the definition (\ref{eq:def-a-u}) of $a_u$,
$$
\mathbb{E}\lp a_u\lp\la  | \ep_{p,N}|\rp\rp\le \mathbb{E}\lp\textup{e}^{\la^{u/2}|\ep_{p,N}|^{u/2}}\rp+\textup{e}^{2/u-1}
$$
for all $p\ge 0,$ all $N\ge 1$ and all $\la>0$.

\smallskip

Then, applying Lemma~\ref{Samy}, we have:
$$
\displaystyle
\mathbb{P}\lp  \widetilde{\Omega}_n^c\cap\lcl \lln R_n \rrn\ge \frac{\hat{\delta}}{2}\rcl  \rp\le \frac{K+\textup{e}^{2/u-1}}{a_u\left(\la M_{n,\hat{\delta}}\right)}
$$
for any $\la \le \ga^{2/u}  N^{1/2}$. Since $M_{n,\hat{\delta}}\ge (1-\rho)\hat{\delta}/2$ and since $a_u$ is a non-decreasing function, by choosing $\la=\ga^{2/u}  N^{1/2}$, we obtain:
$$%\begin{equation}
%\label{Borne2}
\mathbb{P}\lp  \widetilde{\Omega}_n^c\cap\lcl |R_n|\ge \frac{\hat{\delta}}{2}\rcl  \rp \le \frac{K_1}{a_u\lp \ga_1\hat{\delta} N^{1/2}\rp}
$$
%\end{equation}
with $\ga_1=(1-\rho)\ga^{2/u}/2>0$ and $K_1=K+\textup{e}^{2/u-1}$. \\

Choose now $\hat{\delta}=N^{-\al}$, with $\al<1/2$. Observe that  for $N$ large enough,
$\ga_1\hat{\delta}N^{1/2}>s_u$ and thus $a_u\lp \ga_1\hat{\delta}
N^{1/2}\rp=\textup{e}^{\ga_1^{u/2}N^{(1/2-\al)u/2}}$. Hence, there exists a finite positive
constant $K'$ such that for all $N\ge 1$ and $p\ge 0$,
\begin{equation}
\label{Borne2}
\mathbb{P}\lp  \widetilde{\Omega}_n^c\cap\lcl |R_n|\ge \frac{1}{2N^\al}\rcl  \rp \le K'\textup{e}^{-\widetilde{\ga}N^{(1/2-\al)u/2}}
\end{equation}
with $\widetilde{\ga}=\ga_1^{u/2}.$

\smallskip

\noindent {\it Step 3: Conclusion.} Putting together~ (\ref{eq:bnd1-Un-x}), \eqref{majE},
\eqref{Borne1} and \eqref{Borne2}, choosing $\hat{\delta}=N^{-\al}$ with $\al<1/2$, we end up
with:
$$
\mathbb{P}\lp \lln U_n-x^*\rrn\ge N^{-\al} \rp \le nK\textup{e}^{-\ga L^{u/2} N^{u/4}}+K'\textup{e}^{-\widetilde{\ga} \ N^{(1/2-\al)u/2}},
$$
for any $n$ such that $n\ge 1-\al \ln(N/(2g_\infty-2x^*))/\ln( M_{x^*})$.
Choose now  %$\hat{\delta}=N^{-\al}$, with $\al<u/4$, and
 $n=[C\al \ln N]+1$.  If the following condition holds true:
$$
\lim_{N\to+\infty} \left(n+\al \ln(N/(2g_\infty-2x^*))/\ln( M_{x^*})\right)=+\infty
$$
%$$
%\lim_{N\to +\infty}\frac{M_{x^*}^n}{\hat{\delta}}=\lim_{N\to+\infty}\exp\lp \al(1+C\ln(M_{x^*}))N\rp=0,
%$$
i.e. if $C>-1/\ln\lp M_{x^*}\rp$, then for $N_0$ large enough,
$$
n=[C\al \ln N]+1\ge 1-\al \ln(N/(2g_\infty-2x^*))/\ln( M_{x^*}).
$$
%equation~\eqref{controle-n} is fulfilled for any $N\ge N_0$.
We thus choose $C=-1/\ln\lp M_{x^*}\rp+\eta$ with $\eta>0$. Hence, for $N\ge N_0$ and $n=[C\al \ln N]+1$, we have:
$$
\mathbb{P}\lp \lln U_n-x^*\rrn\ge N^{-\al} \rp \le nK\textup{e}^{-\ga L^{u/2} N^{u/4}}+K'\textup{e}^{-\widetilde{\ga} \ N^{(1/2-\al)u/2}}.
$$
Therefore, since $(1/2-\al)u/2\le u/4$  we have proved that there exists a positive finite
constant $A$ such that for all $N\in\N^*$,
$$
\mathbb{P}\lp \lln U_n-x^*\rrn\ge N^{-\al} \rp \le A\textup{e}^{-\widetilde{\ga} N^{(1/2-\al)u/2}},
$$
which is the desired result.
\end{proof}

\subsection{Subcritical case: $\mathbf{F<F_c}$}
We show in this section that the choice of the constant $F_c$ for the peeling algorithm is optimal
in the following sense: if one chooses a parameter $F<F_c$, then the threshold sequence converges
to 0 with high probability. Specifically, we get the following result:
\begin{proposition}\label{Asy:F<Fc}
Consider $F<F_c$ and assume that our signal $z$ satisfies Hypothesis~\ref{hyp:1.1}. Let $N\in
\N^*$, $\al<1/2$, and $n\ge Q_N$, where
$$
Q_N= \max\lp
1+ \frac{\al \ln(N)+\ln(2 g_{\infty})}{\ln{1/\ka}}; \, 1\rp.
$$
Then,  there exist  $A,\widetilde{\ga}$ two positive finite constants (independent of $N$ and $n$) such that
\begin{equation}\label{eq:ineq-P-subcritic}
\mathbb{P}\lp U_n \ge N^{-\al} \rp \le A\textup{e}^{-\widetilde{\ga} N^{(1/2-\al)u/2}}.
\end{equation}
\end{proposition}

\begin{proof}
In the subcritical case, the following property holds true for the function $g\equiv g_{1,u}$
defined by (\ref{eq:relation-f-sigma-b}): there exists a constant $\ka\in (0,1)$ such that, for
all $x\ge 0$, $0\le g(x)\le \ka x$. We thus have the following relation for the noisy dynamics of
$U_n$:
$$
U_{n}=g(U_{n-1})+\ep_{n-1,N}\le \ka U_{n-1}+\ep_{n-1,N}.
$$
Iterating this inequality, we have:
\begin{equation}%\label{eq:ineq-U-n-subcritic}
U_{n} \le \ka^{n-1} U_{1}+ \sum_{j=1}^{n-1} \ka^{j-1} \ep_{n-j,N}.
%= \ka^{n-1} g_\infty + \sum_{j=1}^{n} \ka^{j-1} \ep_{n-j,N},
\end{equation}
According to the fact that $U_1=g_\infty+\ep_{0,N}$, we end up with:
\begin{equation}\label{eq:ineq-U-n-subcritic}
U_{n} \le %\ka^{n-1} U_{1}+ \sum_{j=1}^{n-1} \ka^{j-1} \ep_{n-j,N}.
\ka^{n-1} g_\infty + \sum_{j=1}^{n} \ka^{j-1} \ep_{n-j,N},
\end{equation}
a relation which is valid for any $n\ge 1$.

\smallskip

Consider now $\alpha<1/2$ and assume that $n\ge Q_N$, which ensures
$\ka^{n-1} g_\infty\le N^{-\al}/2$. Then invoking (\ref{eq:ineq-U-n-subcritic}), we have
$$
\mathbb{P}\lp U_n \ge N^{-\al} \rp \le
\mathbb{P}\lp  \sum_{j=1}^{n} \ka^{j-1} \ep_{n-j,N} \ge \frac{N^{-\al}}{2}\rp.
$$
We are thus back to the setting of the proof of Theorem \ref{thm:cvgce-noisy-dyn}, Step 2. Along
the same lines as in this proof (changing just the name of the constants there), the reader can
now easily check inequality (\ref{eq:ineq-P-subcritic}).

\end{proof}

\begin{remark}
We have chosen here to investigate the case of a probability $\mathbb{P}( U_n \ge N^{-\al} )$ and
of a logarithmic number of iterations $n$, in order to be coherent with Theorem
\ref{thm:cvgce-noisy-dyn}. However, in the simpler subcritical setting, one could have considered
a number of iterations of order $N$, opening the door to a possible almost sure convergence of
$U_n$ to 0. We have not entered into those details for sake of conciseness. In the same spirit, we
have not tried to solve the (much harder) problem of the behavior of our algorithm in the critical
case $F=F_c$.
\end{remark}

\section{Denoising algorithms implementation}
The previous sections aimed at giving an optimal criterion of convergence for the peeling
algorithm, in terms of the constant $F_c$, and under the assumption of a signal whose wavelets
coefficients are distributed according to a generalized Gaussian random distribution. We now wish
to test the algorithm we have produced in terms of denoising performances, on an empirical basis.

\smallskip

To this purpose, we shall compare various peeling algorithms (detailed at Section \ref{sec:thres}
below) and two traditional wavelet denoising procedures, namely \textit{Universal} and
\textit{SURE} shrinkage (see \cite{DJ94}). The comparison will be held in two types of situations:
first we consider the benchmark simulated signals proposed in the classical reference \cite{DJ94}.
Then we move to a medical oriented application, by observing the denoising effect of our
algorithms on ECG type signals. In both situations, we shall see that peeling algorithms enable a
good balance between smoothing and preserving the original shape of the noisy signal.

\subsection{Thresholds}\label{sec:thres}
Theorem \ref{thm:cvgce-noisy-dyn} and Remark \ref{rmk:cvgce-noisy-dyn} induce us to implement the
three following procedures:

\smallskip

\noindent \textbf{(1)} The first one exploits only implicitly the peeling approach and can be
reduced to a hard (or soft) thresholding in (\ref{eq:2}), where $T_f$ is obtained in the three
following ways: recall that, according to the value of the shape parameter $u$, we have computed a
critical value $F_c$ above which the peeling algorithm converges to a non trivial limit (see e.g.
Table 1) with high probability. We thus consider two supercritical cases, namely $F_{05}=1.05F_c$
and $F_{15}=1.15F_c$. In these two cases, we compute $\tau^*=(x^*)^{1/2}$, where $x^*$ is the
fixed point of the function $g_{\si,u}$ defined by (\ref{eq:g-sigma-u}), as analyzed at Section
\ref{sec:critic-constants}. We call respectively $T_{c,05}$ and $T_{c,15}$ these two values, which
serve as a threshold in (\ref{eq:2}). A third value of the threshold is also considered by taking
$F=F_m$ in (\ref{eq:g-sigma-u}), where $F_m$ is defined by (\ref{eq:Fm}), and computing then the
corresponding threshold $T_{cm}$. This allows a comparison with the older reference \cite{RHLW2}.
Let us stress the fact that for this first approach, no iterations are performed.

\smallskip

\noindent \textbf{(2)} The second procedure computes the final thresholds using a fixed number of
iterations in the peeling algorithm. According to one of the conclusions in Theorem
\ref{thm:cvgce-noisy-dyn}, we take this number of iterations equal to $\log N$. As in the first
approach, 3 thresholds were obtained, for $F=1.05F_c$, $F=1.15F_c$ and $F=F_m$. Theoretically,
this implementation yields some thresholds $\hat{T}_{c,05}$, $\hat{T}_{c,15}$ and $\hat{T}_{cm}$
which should be close to their respective \emph{exact} counterparts ${T}_{c,05}$, ${T}_{c,15}$ and
${T}_{cm}$ (within the conditions stated by Theorem \ref{thm:cvgce-noisy-dyn}).

\smallskip

\noindent \textbf{(3)} The third implementation is the one proposed in \cite{RHLW2} (fixed point
descent with a sufficient convergence condition $F=F_m$). The  resulting threshold will be noted
as $T_m$.

\smallskip

The relations between the 7 thresholds mentioned above are represented at Figure
\ref{fig:thresall} for different shape parameters $u$. The lines represent the theoretical values
$T_{c,05}$, $T_{c,15}$ and $T_{cm}$, while the shaded zones represent a superposition of the
estimated $\hat{T}_{c,05}$, $\hat{T}_{c,15}$ and $\hat{T}_{c,m}$ obtained over 100 simulations
(generalized gaussian vectors of  $N=10000$ points, zero mean and unitary standard deviation). The
averaged values of these estimations are very close to the theoretical values, which confirms that
peeling algorithms implemented with $\log N$ iterations converge to some thresholds very close to
the theoretical values (see Theorem~\ref{thm:cvgce-noisy-dyn} again). Moreover, the fixed point
implementation taken from \cite{RHLW2} gives almost the same final threshold as $\hat{T}_{cm}$
(the respective curves and shaded zone are merely superposed) and therefore is not figured here.

\begin{figure}
\centering\includegraphics[scale=0.55]{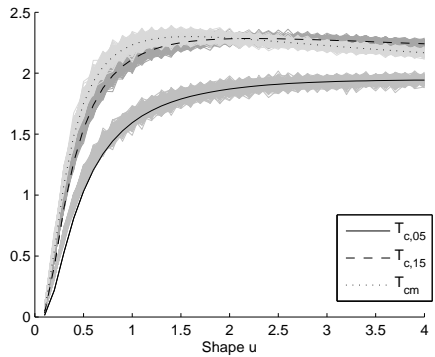}
\caption{Final thresholds for the 7 peeling algorithms for $u=[0.1 \dots 4]$. For comparison,
universal threshold $T_u=4.29$ for $N=10000$.\label{fig:thresall}}
\end{figure}

\subsection{Denoising: simulated signals}
To assess the denoising performances of the peeling algorithm, we used the 4 classical benchmarks
proposed in \cite{DJ94}, namely \textit{Blocks, Bumps, HeaviSine, Doppler} (figure
\ref{fig:donoho}), with 4 lengths ($N=[2048, 4096, 8192, 16384]$). The signals were normalized to
have unitary power. Twenty types of zero-mean random noise $n$ were generated according to
generalized gaussian with shape parameters $u_n=[0.2, 0.4, 0.6 \dots 3.8, 4]$.  The noise was then
scaled to obtain signal to noise ratios $SNR= [1, 2, 3]$, \ie [ 0, 3, 4.8] decibels.
Furthermore, the wavelet decomposition of the signal has been performed based on the {\tt sym 8}
wavelet, and the noise was added to the wavelet coefficients of the noise-free signals to obtain
the ``measured signal'' wavelet coefficients $z$. Each of these noisy signals were simulated 500
times to obtain averaged results. A statistical hypothesis testing showed that the wavelet
coefficient of the signals under consideration could be assimilated to generalized Gaussian random
variables, with the notable exception of the \textit{Bumps} process.

\smallskip

%\begin{figure}
%\subfloat[$Blocks$]{\includegraphics[scale=0.55]{blocks.eps}} \quad
%\subfloat[$Bumps$]{\includegraphics[scale=0.55]{bumps.eps}}\\
%\subfloat[$HeaviSine$]{\includegraphics[scale=0.55]{heavi.eps}} \quad
%\subfloat[$Doppler$]{\includegraphics[scale=0.55]{doppler.eps}}
%\caption{Benchmark signals. \label{fig:donoho}}
%\end{figure}

\begin{figure}
{\includegraphics[scale=0.85]{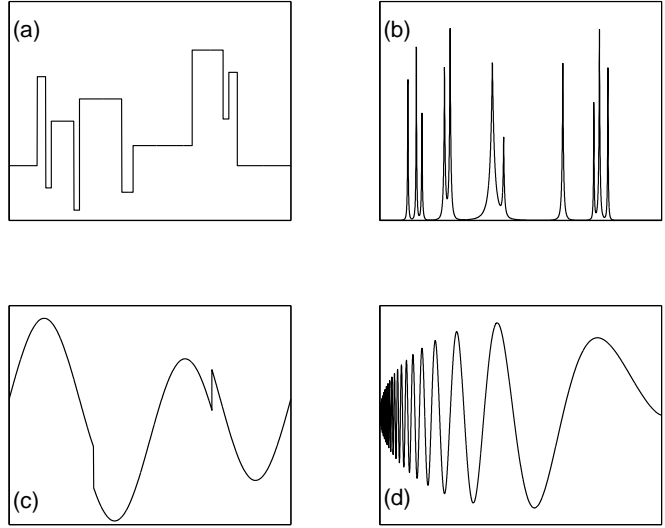}}
\caption{Benchmark signals: (a) \textit{Blocks}, (b) \textit{Bumps},
(c) \textit{HeaviSine,}, (d) \textit{Doppler} \label{fig:donoho}}
\end{figure}

The shape parameter $u_z$ of the total signal $z$, which determines the thresholds of the peeling
algorithms, was estimated using the absolute empirical moments $m_1$ and $m_2$, (with
$m_r=\mathds{E}[|z|^r]$, see~\cite{Mallat89, Kokkinakis05}), while the mean $\mu_z$ and the
standard deviation $\sigma_z$ were estimated using classical empirical estimators.
%For all simulated signals, a Kolmogorov-Smirnov
%test ($\alpha=0.01$) was used to verify that their wavelet coefficients $z$ follow a generalized
%gaussian law $p_{\sigma_z,u_z}(z)$ (\ref{eq:gg-density}). For three signals (\textit{Blocks,
%HeaviSine, Doppler}), the hypothesis of probability laws equality was accepted in more than 95\%
%of the simulations, the notable exception being \textit{Bumps}, for which in more than a half of
%the simulations the equality hypothesis was rejected. In general, for long signals (high $N$) the
%hypothesis was verified.

\smallskip

Denoising was performed by soft thresholding (instead of the hard one described by equation
(\ref{eq:2})) using the 7 algorithms described at Section \ref{sec:thres}, as well as the
classical \textit{Universal} and \textit{SURE} thresholding \cite{DJ94}, for comparison. More
elaborated wavelet denoising methods (either based on redundant wavelet transforms or on block
approaches \cite{Coifman95a, Cai01, Cai08}) were not considered for the comparison, since their
nature is different: all the algorithms tested in this paper are term-by-term approaches for
orthogonal wavelet transform thresholding.

\smallskip

The denoised estimate $\hat{x}$ of the original signal was reconstructed by inverse wavelet
transform. We recall here that \textit{Universal} thresholding aims to completely eliminate
Gaussian noise (and therefore it risks to distort the signal), while \textit{SURE} thresholding
estimates the original signal my minimizing the Stein Unbiased Risk Estimator of the mean squared
error between $x$ and $\hat{x}$, assuming also a Gaussian noise (thus basically aiming a minimum
distortion of the signal, as the peeling algorithms). The denoising performance was evaluated
using the signal to noise ratio after denoising:
$$
SNR_{den}=10\log_{10}\frac{\sum_{i=1}^N(x(i))^2}{\sum_{i=1}^N(x(i)-\hat{x}(i))^2}
$$

\smallskip

As expected, the results obtained for $\hat{T}_{c,05}$, $\hat{T}_{c,15}$, $\hat{T}_{cm}$ and $T_m$
are very similar to those obtained by ${T}_{c,05}$, ${T}_{c,15}$, ${T}_{cm}$, so only results of
the three latter are detailed here \footnote{These three algorithms are of course much faster than
their iterative versions.}. Synthetic comparisons are presented at Figure \ref{fig:denresult} for
different shape parameters $u_n$ of the noise distribution, for all the 4 benchmark signals and
for $N=4096$ (to ease the presentation, detailed tables of results are omitted, since the values can
be read with enough precision on the graphs).
\begin{figure}
\includegraphics[scale=0.55]{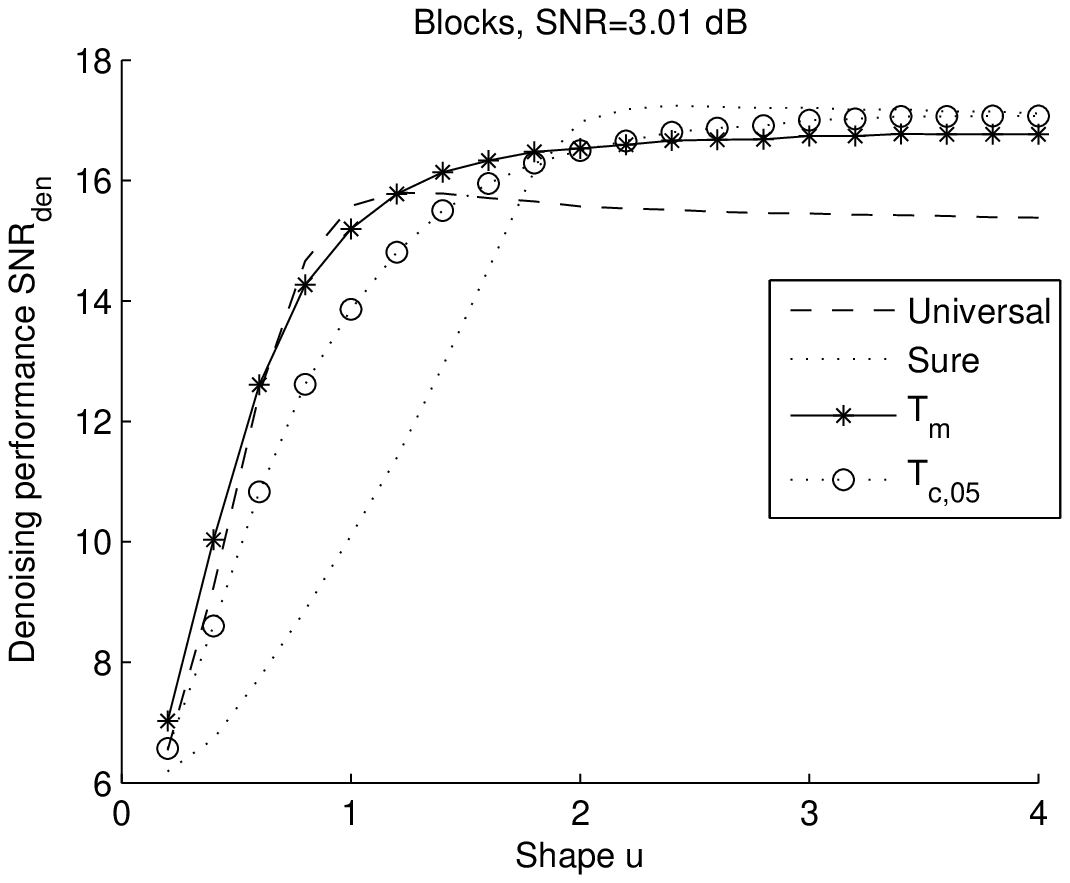} \quad \includegraphics[scale=0.55]{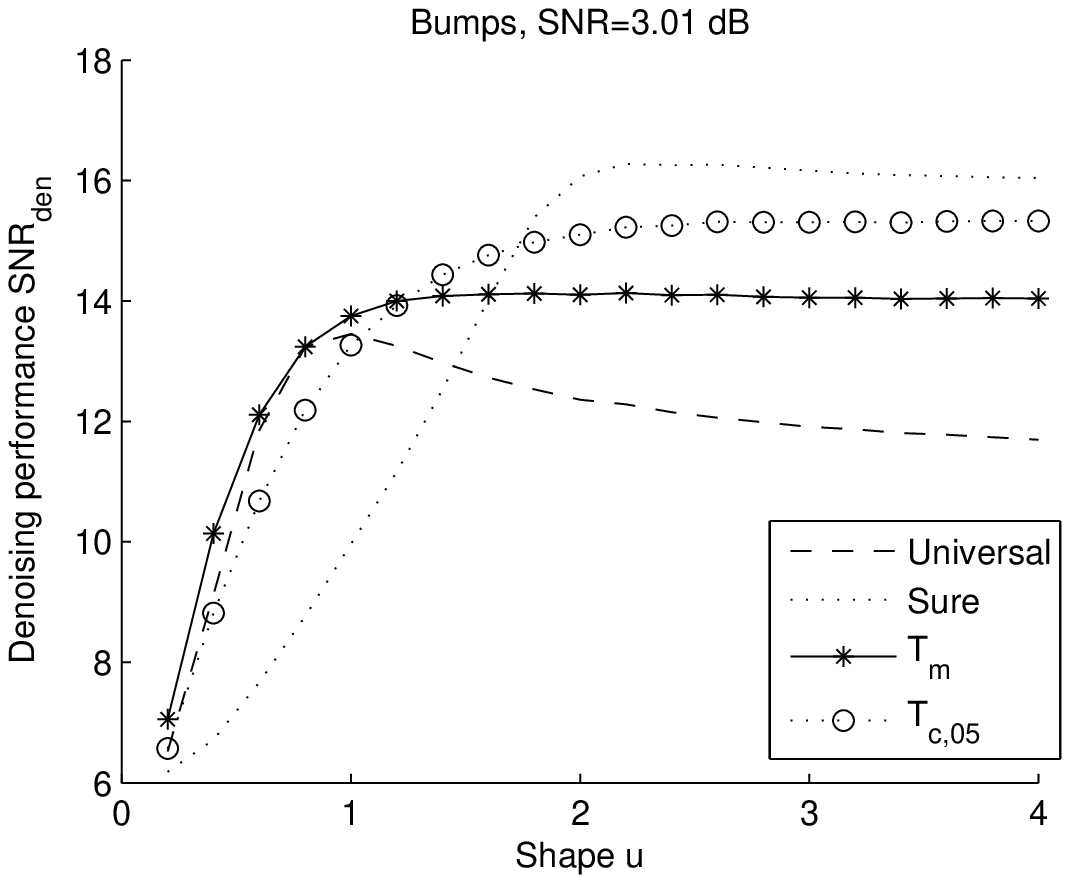}\\
\includegraphics[scale=0.55]{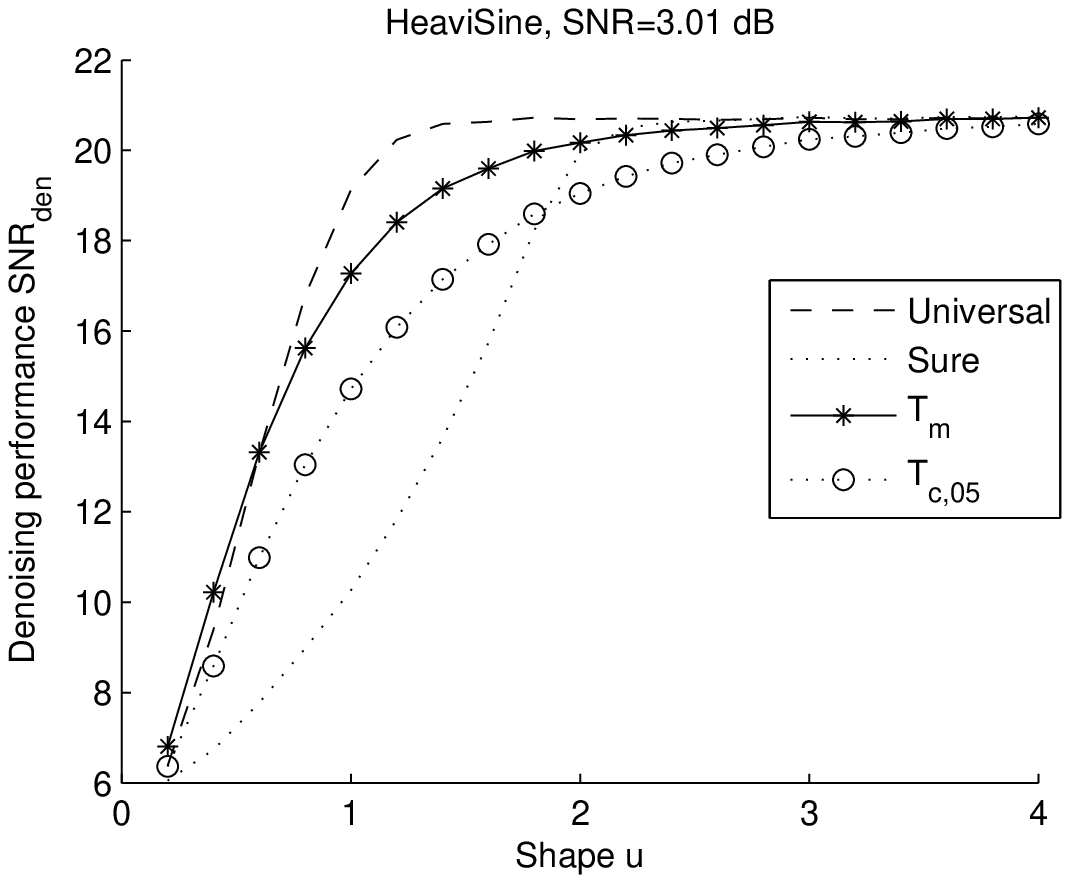} \quad \includegraphics[scale=0.55]{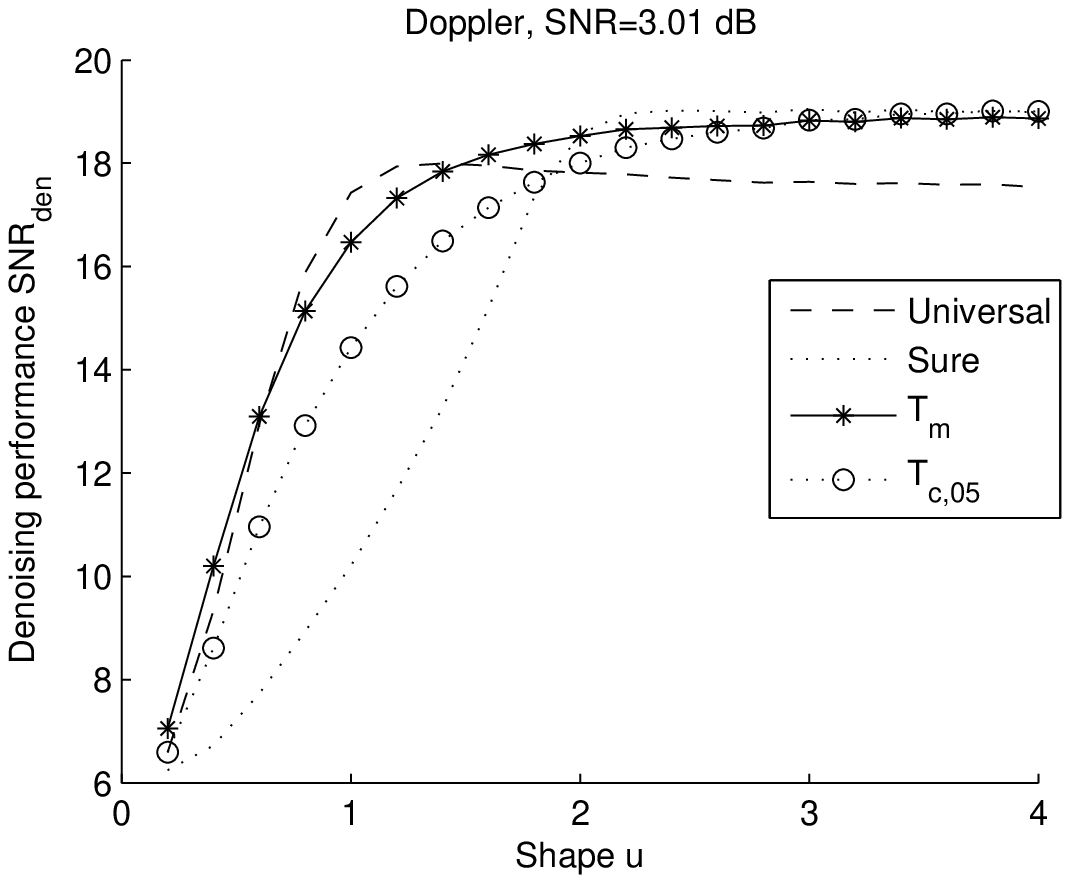}
\caption{Signal to noise ratios after denoising as a function of the noise distribution.
The presented graphs are obtained for noisy signals with $SNR$=3dB. \label{fig:denresult}}
\end{figure}
An illustrative example on the Block signal is also provided at Figure  \ref{fig:blnois}.
\begin{figure}
\includegraphics[scale=0.65]{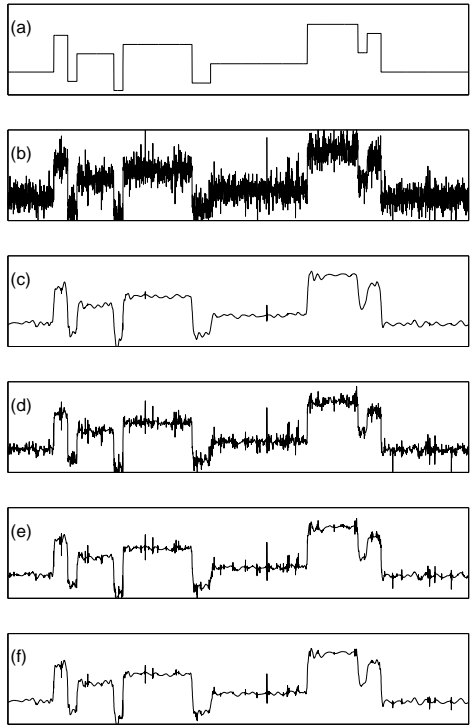}
\caption{Denoising example for $Blocks$: (a) original signal,
(b) noisy signal (Laplacian noise $u_n$=1, $SNR$=10dB), (c) \textit{Universal} ($SNR_{den}$=17.7dB),
(d) \textit{SURE} ($SNR_{den}$=16.9dB), (e) $T_{c,05}$ ($SNR_{den}$=18.6dB), (f) $T_{cm}$ ($SNR_{den}$=18.5dB)
\label{fig:blnois}}
\end{figure}

\smallskip

Several interesting observations can be made. Obviously, the noise type (shape parameter $u_n$)
greatly influences the performances of all algorithms: they are lower for heavy-tailed noise
distributions, which indicates that this type of noise is more difficultly eliminated from
measured signals. As one could expect from its development, \textit{SURE} thresholding is the best
choice for Gaussian noise, for which it also attains its best performance (this algorithm continue
to have a very good performance for higher $u_n$). The \textit{Universal} thresholding attains its
best performance for Laplacian noise $u=1$ and it has very good results for super-Gaussian noises
($u_n < 2$). On the contrary, its performances are the worst for high values of the shape
parameter of the noise (except for the very low frequency signal \textit{HeaviSine}, for which all
algorithms are similar for sub-Gaussian noises $u_n > 2$).

\smallskip

The peeling algorithms need a more detailed analysis: they are better than \textit{SURE} for
super-Gaussian noises, with the theoretical $T_{c,15}$ (or, equivalently, theoretical $T_{cm}$ and
iterative $\hat{T}_{cm}$, $\hat{T}_{c,15}$ and $T_{m}$) being slightly better than $T_{c,0.5}$.
The order between the two peeling algorithms tend to change for sub-Gaussian noise ($u_n > 2$),
especially for impulsive \textit{Blocks} and \textit{Bumps}. To conclude, it seems that for
super-Gaussian noises, \textit{Universal} thresholding and peeling algorithms are the best choice,
while for sub-Gaussian noises the results are almost similar, with \textit{SURE} thresholding
having the best performances when the noise is almost Gaussian. In all, the peeling algorithm with
$T_{c,15}$ works in a satisfying way, independently of the shape parameter $u$.

\subsection{Denoising: real and pseudo-real signals}

A last point should be reminded: peeling algorithms were mainly developed for biomedical
applications \cite{CW,HP}. Therefore, we have chosen to evaluate the performance of the newly
developed versions on  biological signals (normal electrocardiogram -- ECG and normal background
electroencephalogram EEG). However, when dealing with real signals for denoising, one is faced
with the following problem: it is impossible to assert that the denoising is accurate when the
original signal is unknown. It is also very hard to be provided with a non noisy signal which can
be perturbed artificially.

\smallskip

In order to cope with this situation, we have chosen to work with a commonly used ECG simulator,
implemented in Matlab. In this way, one can produce a clean ECG type signal, spoil it with an
artificial noise, and then try to recover the original signal by some denoising procedures. The
simulated EEG was generated according to the procedure described in \cite{Rankine07} (see the url
for the Matlab code).

\smallskip

This is the protocol we have followed for our experiment. Numerical results globally confirm those
obtained on the benchmark signals, both for the simulated ECG and EEG. They are not reproduced
here for sake of conciseness, but an illustrative example is given at Figure \ref{fig:ECG}.
\begin{figure}
\includegraphics[scale=0.75]{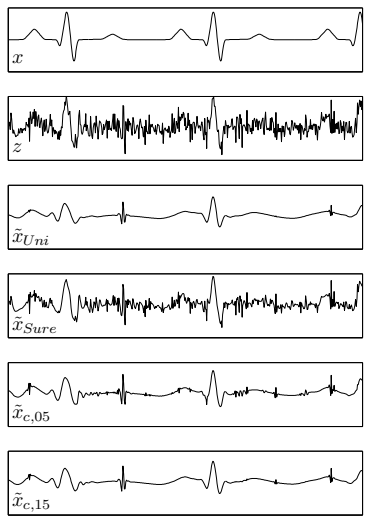}
\caption{Denoising example for a simulated $ECG$: $x$ original signal,
$z$ noisy signal (Laplacian noise $u_n$=1, $SNR$=0dB), $\hat{x}_{Uni}$ \textit{Universal} thresholding ($SNR_{den}$=4.8dB),
$\hat{x}_{Sure}$ \textit{Sure} thresholding ($SNR_{den}$=3.7dB),
$\hat{x}_{c,05}$ $T_{c,05}$ thresholding ($SNR_{den}$=5.9dB) and
$\hat{x}_{c,15}$ $T_{c,15}$ thresholding ($SNR_{den}$=5.3dB).
\label{fig:ECG}}
\end{figure}

Two-dimensional versions of the tested algorithms were applied on real benchmark images also
($Lena$, $House$, $Barbara$, $Peppers$), with similar performances to those obtained for the 1-D
signals. Therefore, the detailed results are not presented here.

%\begin{figure}
%\subfloat[]{\includegraphics[scale=0.7]{lena.pdf}}
%\subfloat[]{\includegraphics[scale=0.7]{lenab.pdf}}
%\subfloat[]{\includegraphics[scale=0.7]{lenau.pdf}}\\
%\subfloat[]{\includegraphics[scale=0.7]{lenas.pdf}}
%\subfloat[]{\includegraphics[scale=0.7]{lenac05.pdf}}
%\subfloat[]{\includegraphics[scale=0.7]{lenam.pdf}}
%\caption{Denoising example for $Lena$: (a) original signal,
%(b) noisy signal (Laplacian noise $u_n$=1, $SNR$=18.4dB/ $PSNR$=24.1dB),
%(c) \textit{Universal} ($SNR$=21.2dB/ $PSNR$=26.9dB),
%(d) \textit{SURE} ($SNR$=22.4dB/ $PSNR$=28.1dB), (e) $T_{c,05}$ ($SNR$=23.4dB/ $PSNR$=29.1dB),
%(f) $T_{cm}$ ($SNR$=22.3dB/ $PSNR$=28.0dB)
%\label{fig:lena}}
%\end{figure}


\begin{thebibliography}{99}

\bibitem{ABS}
A. Antoniadis, J. Bigot, T. Sapatinas:
Wavelet estimators in non-parametric regression: a comparative simulation study.
Preprint, 2006.

\bibitem{Ch}
C. Chesneau:
Wavelet block thresholding for samples with random design: a minimax approach under the $L^p$ risk.
{\it Electron. J. Stat.} {\bf 1} (2007), 331--346.

\bibitem{Coifman95a} R. Coifman, D. Donoho: Translation invariant denoising. {\it Wavelets and
    Statistics}, ed. A. Antoniadis and G. Oppenheim, Springer Verlag, 125-150 (1995).

\bibitem{CW} R. Coifman, M. Wickerhauser: Adapted waveform de-noising for medical signals an
    images, {\it IEEE Engineering in Medecine and Biology Magazine} {\bf 14}, no. 5, 578-586
    (1995).

\bibitem{Da}
I. Daubechies:
{\it Ten lectures on wavelets.}
Society for Industrial and Applied Mathematics (1992).

\bibitem{DJ94}
D. Donoho, I. Johnstone:
Ideal spatial adaptation via wavelet shrinkage. {\it
    Biometrika} {\bf 81}, 425-455 (1994).


\bibitem{DJKP}
D. Donoho, I. Johnstone, G. Kerkyacharian, D. Picard:
Wavelet shrinkage: asymptopia? With discussion and a reply by the authors.
{\it J. Roy. Statist. Soc. Ser. B}  {\bf 57},  no. 2, 301--369 (1995).

\bibitem{GN}
E. Giné, R. Nickl:
Uniform limit theorems for wavelet density estimators.
{\it Ann. Probab.} {\bf 37}, no. 4, 1605--1646 (2009).

\bibitem{HP} L. Hadjileontiadis, S. Panas: Separation of discontinuous adventitious sounds from
    vesicular sounds using a wavelet based filter. {\it IEEE Transactions on Biomedical
    Engineering} {\bf 44}, no. 12, 1269-1281 (1997).

\bibitem{Kokkinakis05}
K. Kokkinakis, A. Nandi:
Exponent parameter estimation for generalized
    Gaussian probability density functions with application to speech modeling.
    {\it Signal Processing} {\bf 85}, 1852-1858 (2005).

\bibitem{Mallat89}
S. Mallat:
A Theory for Multiresolution Signal Decomposition: The Wavelet
    Representation. {\it IEEE Transactions on Pattern Analysis and Machine Intelligence} {\bf 11},
    no. 7, 674-693 (1989).

\bibitem{Mal}
S. Mallat:
{\it A wavelet tour of signal processing.}
Academic Press (1997).

\bibitem{PZP}
A. Pizurica, V. Zlokolica, and W. Philips:
Noise reduction in video sequences using wavelet-domain and temporal filtering,
in SPIE Conference \textit{Wavelet Applications in Industrial Processing}, Providence, Rhode Island, USA (2003).


\bibitem{RHLW}
R. Ranta, C. Heinrich, V. Louis-Dorr, D. Wolf:
Interpretation and improvement of an iterative wavelet-based denoising method.
{\it IEEE Signal Processing Letters} {\bf 10}, no. 8, 239-241 (2003).

\bibitem{RHLW2}
R. Ranta, C. Heinrich, V. Louis-Dorr, D. Wolf:
Iterative wavelet-based denoising methods and robust outlier detection.
{\it IEEE Signal Processing Letters} {\bf 12}, no. 8, 557-560 (2005).

\bibitem{VW}
A. van der Vaart, J. Wellner:
{\it Weak convergence and empirical processes},
Springer (1996).

\bibitem{Cai01} T. Cai, B. Silverman: Incorporating information on neighbouring
    coefficients into wavelet estimation,{\it Sankhy\={a}: The Indian Journal of Statistics. Special issue on
    Wavelets} {\bf 63}, no. 2, 127-148 (2001).

\bibitem{Cai08} T. Cai, H. Zhou: A Data-Driven Block Thresholding Approach to Wavelet Estimation,
    {\it The Annals of Statistics}, ~to appear, http://www-stat.wharton.upenn.edu/~tcai/Papers.html

\bibitem{Rankine07} L. Rankine, N. Stevenson, M. Mesbah, B. Boashash: A Nonstationary Model of
    Newborn EEG, {\it IEEE Transactions on Biomedical Engineering} {\bf 54}, no. 1, 19-28 (2007),
    http://www.som.uq.edu.au/research/sprcg/newborn.asp

\end{thebibliography}
\end{document}